\documentclass[letterpaper, 10 pt]{ieeeconf}  
\usepackage{amssymb}
\usepackage{amsmath}
\usepackage{graphicx}
\usepackage{comment}
\usepackage{booktabs}
\usepackage{bm}
\usepackage{graphicx}
\usepackage{epstopdf}
\usepackage{amssymb}
\usepackage{wasysym}
\usepackage{float}
\usepackage{color}
\usepackage{dsfont}
\usepackage{algorithm,algorithmicx}
\usepackage{algpseudocode}
\usepackage{epstopdf}
\usepackage{tabularx}
\usepackage{cite}
\usepackage{graphicx}
\usepackage[caption=false]{subfig}
\let\labelindent\relax
\usepackage{enumitem}
\usepackage{fancyhdr}
\usepackage{marginnote}
\newtheorem{theorem}{Theorem}

\newtheorem{definition}{Definition}

\newtheorem{remark}{Remark}
\newtheorem{assumption}{Assumption}
\newtheorem{lemma}{Lemma}
\newtheorem{alg}{Algorithm}
\graphicspath{ {./img/} }
\newcommand{\eod}{\ensuremath{\hfill\Box}}
\IEEEoverridecommandlockouts                              
\title{\LARGE \bf
Distributed Augmented Lagrangian Method for Link-Based Resource Sharing Problems of Multi-Agent Systems 
}
\author{Wicak Ananduta, Angelia Nedi\'{c}, and Carlos Ocampo-Martinez
	\thanks{W. Ananduta is with the Delft Center for Systems and Control, TU Delft, Netherlands. C. Ocampo-Martinez is with Institut de Rob\`{o}tica i Inform\`{a}tica Industrial
		(CSIC-UPC), Barcelona, Spain. A. Nedi\'{c} is with School of Electrical, Computer and Energy Engineering, Arizona State University, Arizona, USA (emails: {\tt\small w.ananduta@tudelft.nl, carlos.ocampo@upc.edu, angelia.nedich@asu.edu})}
	\thanks{This work has received funding from the European Union's Horizon 2020 research and innovation programme under the Marie Sk\l{}odowska-Curie grant agreement No 675318 (INCITE). }
}
\begin{document}

\maketitle
\thispagestyle{plain}
\pagestyle{plain}

\begin{abstract}
A multi-agent optimization problem motivated by the management of energy systems is discussed. The associated cost function is separable and convex although not necessarily strongly convex and there exist edge-based coupling equality constraints. In this regard, we propose a distributed algorithm based on solving the dual of the augmented problem. Furthermore, we consider that the communication network might be time-varying and the algorithm might be carried out asynchronously. The time-varying nature and the asynchronicity are modeled as random processes. Then, we show the convergence and the convergence rate of the proposed algorithm under the aforementioned conditions. 
\end{abstract}
\vspace{-5pt}
\section{Introduction}
We consider an optimization problem of multi-agent systems. Specifically, the agents in the network cooperatively optimize a separable convex cost function subject to convex local constraints and equality coupling constraints. Furthermore, the set of decision variables of each agent is partitioned into shared decisions, i.e., variables that are involved in the coupling constraints and can be shared with other agents, and private decisions, i.e., variables that must be kept private and cannot be shared with other agents. The objective of the paper is to develop a distributed method for this problem. In addition, we also consider imperfect operation where the communication network might be time-varying and the algorithm might not be carried out synchronously.


The optimization problem considered is mainly motivated by economic dispatch problems of large-scale energy systems \cite{molzahn2017} and belongs to a subclass of network flow problems \cite{bertsekas1998}, where we seek an optimal flow of  certain goods from some sources to some sinks. Specifically, it is a convex network flow problem with a particular control structure, where each node has a computational unit and these units cooperatively solve the problem of the network. As a convex network flow problem, it represents an optimization problem of flow-based networks, such as  electrical \cite{molzahn2017}, thermal energy \cite{rostampour2019} and water networks \cite{grosso2014}. Moreover, the problem also represents a convex relaxation of network flow problems of indivisible goods. Solutions to a convex relaxation  can be used as a lower bound of the optimal solution to the original problem. \color{black}

{The first challenge of developing a distributed method to solve the problem considered is how to deal with the coupling constraints. To address this issue, we employ 
the Lagrangian relaxation  \cite{bertsekas1997,boyd2010,boyd2011}.} The main idea of this concept is to relax the coupling constraints such that the relaxed problem is decomposable. In this regard, Lagrange multipliers associated with the coupling constraints are introduced. In the dual problem, we aim to maximize such multipliers. Many distributed optimization methods, including those for energy management problems, e.g.,\cite{bakirtzis2003,kraning2014,wang2015,hans2018}, are developed based on solving the dual problem. 
Such distributed algorithms are iterative and require the exchanging of information. In particular, the agents that are coupled through link-based constraints must communicate certain information at each iteration. In this article, we consider the case when the information exchange process might be imperfect. In particular, we study the possibility of having a randomly time-varying communication network and asynchronous updates, which are 
relevant to the applications that we consider \cite{ge2017}. \color{black} 

Therefore, in this paper, we develop a distributed optimization algorithm suitable for the previously explained problem. The algorithm is based on the Lagrange dual approach. Furthermore, we also consider that the communication network might be stochastically time-varying and the algorithm can be implemented asynchronously. 
Then, we show analytically that the sequence generated by the proposed algorithm converges to an optimal solution almost surely with the rate of $\mathcal{O}(1/k)$. 

{Now, we position the contributions of this work with respect to the existing literature.}  As previously mentioned, the problem that we consider is suitable to be decomposed using the Lagrangian dual approach. Although it is possible to reformulate the problem into a consensus-based problem \cite{nedic2015}, the latter approach can become impractical when the number of agents is large because the information that must be exchanged is unnecessarily large. 
In order to deal with a larger class of cost functions, particularly those that are not necessarily strongly convex, we consider augmenting the problem. 
In this regard, the proposed algorithm is more closely related to the accelerated distributed augmented Lagrangian (ADAL) method, discussed in \cite{chatzipanagiotis2015,lee2018}, than to the alternating direction method of multipliers (ADMM) \cite{bertsekas1997,boyd2011}. Similar to the ADAL method, we use some information from the neighbors in the local optimization step and require a convex combination step to update the primal variable. Differently, in the proposed method, each agent only performs a convex combination step to update the shared variables instead of all the decisions. Moreover, since we consider a different augmented Lagrangian function, the condition of the step size, which guarantees convergence, is also different. Finally, we also note that the ADAL method in \cite{chatzipanagiotis2015,lee2018} considers perfect communication, i.e., with a fixed neighbor-to-neighbor communication graph.

 In this work, we are interested in developing a distributed method that works over a time-varying communication graph and asynchronous updates, implying imperfect information exchanges.   
To that end, we consider the communication network and asynchronous updates as random processes, similarly to the work in  \cite{tsai2017,wei2013,chang2016,hong2017}. 
It is important to note that the distributed algorithm developed in this manuscript is different from those in \cite{tsai2017,wei2013,chang2016,hong2017} as they consider the ADMM approach and, to the best of our knowledge, the ADAL approach that we consider has not been employed on stochastically time-varying networks.   
In addition, for energy management problems, distributed methods that have been proposed, e.g., \cite{kraning2014,wang2015,hans2018,kar2014,hug2015,baker2016,kargarian2018,sousa2019,ananduta2020} typically assume a perfect communication process, i.e., the necessary information required to execute the updates is available at each iteration. Therefore, for the considered applications, this technical note provides a more resilient distributed method than those in the aforementioned papers 
in dealing with potential communication problems. \color{black}



\section{Problem formulation}
\label{sec:prob_form}
\subsection{Preliminaries}
\paragraph*{Notations} We consider all vectors as column vectors. A stack of column vectors $x_i$, for all $i\in\mathcal{N}=\{1,2,\dots,n\}$, denoted by $[x_i]_{i\in\mathcal{N}}$, is also a column vector. The inner product of vectors $x$ and $y$ is denoted by $\langle x,y\rangle$. The Euclidean norm of vector $x$ is denoted by $\|x\|_2$. Moreover, for a diagonal matrix $D \in \mathbb{R}^{d\times d}$, the square of a weighted vector norm induced by $D$ is denoted by $\|\cdot\|_D^2$, i.e., $\|x\|_D^2=\langle x,Dx\rangle$, for a vector $x$.  The all-ones vector with the size of $n$ is denoted by $\mathds{1}_n$ whereas the identity matrix with the size $n\times n$ is denoted by $I_n$. Furthermore, the block-diagonal operator, which construct a block diagonal matrix of the arguments, is denoted by $\operatorname{blkdiag}(\cdot)$. 
\begin{definition}
	\label{def:cvx}
	\textit{(Convexity)} A differentiable function $f:\mathbb{R}^n \to \mathbb{R}$ is convex, if, for any $x,y\in\mathbb{R}^n$, it holds that
	$$f(y)-f(x)\geq \langle\nabla f(x),y-x\rangle.$$
\end{definition}
\begin{definition}
	\label{def:strong_cvx}
	\textit{(Strong convexity \cite[Theorem 5.24.iii]{beck2017})} A differentiable function $f:\mathbb{R}^n\to\mathbb{R}$ is strongly convex with strong convexity constant $m$, if, for any $x,y\in\mathbb{R}^n$, it holds that
	$$ \langle\nabla f(y)-\nabla f(x),y-x\rangle \geq m \|y-x\|^2.$$
\end{definition}

\subsection{Multi-agent optimization}
We consider a group of interconnected agents that is represented by an undirected graph $\mathcal{G} = (\mathcal{N},\mathcal{E})$, where $\mathcal{N}=\{1,\dots,n \}$ denotes the set of agents and $\mathcal{E} \subseteq \mathcal{N}\times\mathcal{N}$ denotes the set of links that connect the agents, i.e., $\{i,j\}\in \mathcal{E}$ means that agent $i$ is coupled with agent $j$ in a constraint. Furthermore, denote the set of neighbors of agent $i$ by $\mathcal{N}_i=\{j: \{i,j\} \in \mathcal{E} \}$. The optimization problem that all agents consider to solve cooperatively is
\begin{subequations}
\begin{align}
&\underset{(\bm{u}_i,\bm{v}_i) \in \mathcal{C}_i,i\in\mathcal{N}}{\text{minimize}} \quad  \sum_{i=1}^{n} \left(f_i^{\mathrm{p}}(\bm{u}_{i}) + f_i^{\mathrm{s}}(\bm{v}_{i})\right) \label{eq:cost_f_t}\\
&\text{subject to } \bm{v}_{i}^j + \bm{v}_{j}^i = 0, \quad \forall j\in\mathcal{N}_i, \ i \in \mathcal{N}, \label{eq:coup_pb}
\end{align}
\label{eq:opt}%
\end{subequations}
where each agent $i$ has private/local decisions denoted by $\bm{u}_i \in  \mathbb{R}^{n_i^{\mathrm{p}}}$ and a shared decision denoted by $\bm{v}_{i}^j \in \mathbb{R}^{n_{\mathrm{s}}}$, for every neighbor $j \in \mathcal{N}_i$. The vector $\bm{v}_i$ collects all the shared decisions of agent $i$, i.e., $\bm{v}_i=[\bm{v}_{i}^j]_{j\in\mathcal{N}_i}$. For each agent $i$, the cost function in \eqref{eq:cost_f_t} is divided into two parts:  $f_i^{\mathrm{p}}$ and $f_i^{\mathrm{s}}$, which depend on $\bm{u}_{i}$ and $\bm{v}_{i}$, respectively.  Furthermore, the decisions of agent $i$ $(\bm{u}_i,\bm{v}_i)$  are constrained by the local set $\mathcal{C}_i$. Moreover, the shared decisions of agent $i$ are also coupled with the shared decisions of its neighbors through the equality constraints \eqref{eq:coup_pb}. 
Additionally, 
 we suppose that the following assumptions hold.
\begin{assumption}
	\label{as:strict_cvx_cost}
	The functions $f_i^{\mathrm{p}}:\mathbb{R}^{n_i^{\mathrm{p}}} \to \mathbb{R}$ and $f_i^{\mathrm{s}}:\mathbb{R}^{n_{\mathrm{s}}|\mathcal{N}_i|} \to \mathbb{R}$, for each $i\in\mathcal{N}$, are differentiable and convex. Moreover, $f_i^{\mathrm{p}}(\bm{u}_{i})$, for each $i\in\mathcal{N}$, is strongly convex with strong convexity constant, denoted by $m_i$. \eod
\end{assumption}
\begin{assumption}
	The set $\mathcal{C}_i$, for each $i\in\mathcal{N}$, is polyhedral and compact. \eod
	\label{as:comp_loc_set}
\end{assumption} 
\begin{assumption}
	The feasible set of Problem \eqref{eq:opt} is non-empty. \eod
	\label{as:nonempty_set}
\end{assumption}
\begin{remark}
	By Assumption \ref{as:strict_cvx_cost}, the cost function is continuous. Based on the Weierstrass theorem, since the problem is feasible (Assumption \ref{as:nonempty_set}) and  $\mathcal{C}_i$, for each $i\in\mathcal{N}$, is compact, the optimal value is finite and the problem has a solution. \eod
\end{remark}
\begin{remark}
As practical examples, we refer to \cite{kar2014,hug2015,baker2016,kargarian2018,sousa2019,ananduta2020} for energy management problems that consider the same problem structure, i.e., polyhedral and compact local set constraints and edge-based coupling constraints. \eod 
\end{remark}

\subsection{Time-varying communication and asynchronicity}
We model the communication network as a random graph \cite{wei2013}. To that end, let the communication network be described as an undirected graph $\mathcal{G}^c(k) = (\mathcal{N},\mathcal{E}^c(k))$, where $\mathcal{E}^c(k) \subseteq \mathcal{E}$ denotes the set of communication links that are active at iteration $k{-1}$, i.e., $\{i,j\} \in \mathcal{E}^c(k)$ means that agents $i$ and $j$ can exchange information between each other. Thus, the random model of the communication network is defined in Assumption \ref{as:prob_link}.
\begin{assumption}
	The set $\mathcal{E}^c(k)$ is an independent and identically distributed  random variable. Furthermore, any communication link between two coupled agents $i$ and $j$, where $\{i,j\} \in \mathcal{E}$, is active with a positive probability denoted by $\beta_{ij}$, i.e., $\mathbb{P}\left(\{i,j\} \in \mathcal{E}^c(k)\right)=\beta_{ij}>0.$ \eod
	\label{as:prob_link}
\end{assumption}

Moreover, we also allow asynchronous updates, i.e., not all agents might update their decisions at each iteration. The asynchronous updates are also modeled as a random process, as follows. Denote the set of agents that are active and update their primal and dual variables at iteration $k{-1}$ by $\mathcal{A}(k)$. Then, we consider the following assumption.
\begin{assumption}
	The set $\mathcal{A}(k) \subseteq \mathcal{N}$ is an independent and identically distributed random variable. Moreover, an agent $i \in \mathcal{N}$ is active and updates its primal and dual variables at  iteration $k$ with a positive probability denoted by $\gamma_{i}$, i.e., $\mathbb{P}\left(i \in \mathcal{A}(k)\right)=\gamma_{i}>0.$\eod
	\label{as:act_agent}
\end{assumption}

\section{Proposed method}
\label{sec:prop_al}

\subsection{Algorithm design}

We consider the augmented problem of \eqref{eq:opt} in the following form:
\begin{align}
&\underset{(\bm{u}_i,\bm{v}_i) \in \mathcal{C}_i,i\in\mathcal{N}}{\text{minimize}} \  \sum_{i=1}^{n} \left(f_i^{\mathrm{p}}(\bm{u}_{i}) + f_i^{\mathrm{s}}(\bm{v}_{i})+\sum_{j\in \mathcal{N}_i}\|\bm{v}_{i}^j + \bm{v}_{j}^i\|_2^2\right) \notag\\
&\text{subject to } \bm{v}_{i}^j + \bm{v}_{j}^i = 0, \quad \forall j\in\mathcal{N}_i, \ i \in \mathcal{N}. \label{eq:opt2}
\end{align}
We use the dual approach to decompose Problem \eqref{eq:opt2}. To this end,  we denote the decisions of all agents by $\bm{u}=[\bm{u}_i]_{i\in\mathcal{N}}$ and  $\bm{v}=[\bm{v}_i]_{i\in\mathcal{N}}$ \color{black} and we introduce the Lagrangian of the augmented problem \eqref{eq:opt2}, denoted by $L(\bm{u},\bm{v},\bm{\lambda})$, as follows:
\begin{equation}
	\begin{aligned}
	L(\bm{u},\bm{v},\bm{\lambda}) &= \sum_{i\in\mathcal{N}}\Bigl(f_i^{\mathrm{p}}(\bm{u}_{i}) + f_i^{\mathrm{s}}(\bm{v}_{i})+\Bigr.\\
	&\qquad \Bigl. \sum_{j\in \mathcal{N}_i}\left(\langle \bm{\lambda}_i^j,\bm{v}_{i}^j+\bm{v}_{j}^i\rangle+\|\bm{v}_{i}^j + \bm{v}_{j}^i\|_2^2\right)\Bigr),
	\end{aligned}
	\label{eq:aug_Lagr}
\end{equation}
where the coupled constraints in \eqref{eq:coup_pb} are relaxed and $\bm{\lambda}_i^j \in \mathbb{R}^{n_{\mathrm{s}}}$, for all $j\in\mathcal{N}_i$, are the Lagrange multipliers associated to them. Note that, for convenience,  the Lagrange multipliers are compactly written as $\bm{\lambda} = [\bm{\lambda}_i]_{i\in\mathcal{N}}$, where  $\bm{\lambda}_i=[\bm{\lambda}_i^j]_{j\in\mathcal{N}_i}$. Now, we introduce the dual function, denoted by $q(\bm{\lambda})$, as follows:
\begin{equation}
q(\bm{\lambda}) = \underset{(\bm{u}_i,\bm{v}_i) \in \mathcal{C}_i,i\in\mathcal{N}}{\text{minimize}} L(\bm{u},\bm{v},\bm{\lambda}).  \label{eq:dual_f}%
\end{equation}%
 Since each $\mathcal{C}_i$ is assumed to be compact (Assumption \ref{as:comp_loc_set}) and the Lagrangian function is continuous (Assumption \ref{as:strict_cvx_cost}), by the Weierstrass theorem it follows that a minimizer in \eqref{eq:dual_f} exists and the value $q(\bm{\lambda})$ is finite for every $\bm{\lambda}$. \color{black} Hence, the domain of $q(\bm{\lambda})$ is the entire space of $\bm{\lambda}$. We also know from the duality theory that $q(\bm{\lambda})$ is concave and continuous.
 
 \smallskip
The dual problem associated with \eqref{eq:opt2} is stated as follows: 
\begin{equation}
\underset{\bm{\lambda}}{\text{maximize}} \ q(\bm{\lambda}).
\label{eq:max_min_La}
\end{equation}
Note that the dual optimal value is finite. Furthermore, the strong duality holds and the set of dual optimal points is non-empty since, in the primal problem \eqref{eq:opt}, the cost function is convex and the constraints are linear \cite[Proposition 5.2.1]{bertsekas1995}. In other word, there exists a saddle point of the Lagrangian function $L(\bm{u},\bm{v},\bm{\lambda})$,  i.e., a point $(\bm{u}^{\star},\bm{v}^{\star},\bm{\lambda}^{\star}) \in \prod_{i\in\mathcal{N}} \mathcal{C}_i \times \mathbb{R}^{\sum_{i\in\mathcal{N}}n_s|\mathcal{N}_i|}$ such that, for any $(\bm{u},\bm{v}) \in \prod_{i\in\mathcal{N}} \mathcal{C}_i$ and $\bm{\lambda}\in\mathbb{R}^{\sum_{i\in\mathcal{N}}n_s|\mathcal{N}_i|}$, it holds that
\begin{equation}
	L(\bm{u}^{\star},\bm{v}^{\star},\bm{\lambda}) \leq L(\bm{u}^{\star},\bm{v}^{\star},\bm{\lambda}^{\star}) \leq L(\bm{u},\bm{v},\bm{\lambda}^{\star}).
	\label{eq:saddle_point}
\end{equation} 

The dual function $q(\bm{\lambda})$ has separable constraints and all the terms in the Lagrangian function are also separable, except for the quadratic term $\sum_{i\in \mathcal N}\sum_{j \in \mathcal N_i}\|\bm{v}_{i}^j + \bm{v}_{j}^i\|_2^2$. In this regard, each agent will use the information from its neighbors as a way to approximate that quadratic term and decompose $q(\bm{\lambda})$. 
For each agent $i\in\mathcal{N}$, denote by $\tilde{\bm{v}}_{j}^i$ the information associated to ${\bm{v}}_{j}^i$ from neighbor $j \in \mathcal{N}_i$. Thus, the minimization on the right hand side of \eqref{eq:dual_f} is approximated by
\begin{align}
&\underset{(\bm{u}_i,\bm{v}_i) \in \mathcal{C}_i,i\in\mathcal{N}}{\text{minimize }} \sum_{i\in\mathcal{N}}\Bigl(f_i^{\mathrm{p}}(\bm{u}_{i}) + f_i^{\mathrm{s}}(\bm{v}_{i})+\Bigr. \notag\\
&\qquad \qquad \qquad \Bigl.\sum_{j\in \mathcal{N}_i}\left(\langle \bm{\lambda}_i^j+\bm{\lambda}_j^i,\bm{v}_{i}^j\rangle+\|\bm{v}_{i}^j + \tilde{\bm{v}}_{j}^i\|_2^2\right)\Bigr). \label{eq:q_tilde}	
\end{align}


\begin{algorithm}[t]
\begin{alg}Distributed augmented Lagrangian\hfill
	\label{alg:std}
\smallskip
\hrule
\smallskip
	\noindent	\textbf{Initialization:} For each agent $i\in\mathcal{N}$, $\bm{v}_i(0)=\bm{v}_{i0} \in \mathbb{R}^{|\mathcal{N}_i| n_s}$ and $\bm{\lambda}_i(0) = \bm{\lambda}_{i0}  \in \mathbb{R}^{|\mathcal{N}_i| n_s}$.
	
	\noindent \textbf{Iteration:} For each agent $i\in \mathcal{N}$, 
	\begin{enumerate}
		\item Update $\bm{u}_i(k+1)$ and $\hat{\bm{v}}_i(k)$ according to
		\begin{align}
		&\left(\bm{u}_i{(k+1)},\hat{\bm{v}}_i{(k)}\right) \notag\\
		&= \arg \min_{(\bm{u}_i,\bm{v}_i) \in \mathcal{C}_i} f_i^{\mathrm{p}}(\bm{u}_{i})+f_i^{\mathrm{s}}(\bm{v}_{i})+ \label{eq:primal_uv_up_agl}\\
		&\qquad \sum_{j\in \mathcal{N}_i} \left(  \langle\bm{\lambda}_i^j(k)+\bm{\lambda}_j^i(k),\bm{v}_{i}^j\rangle + \|\bm{v}_{i}^j + \bm{v}_{j}^i(k) \|_2^2\right). \notag
		\end{align}
		\item Update ${\bm{v}}_{i}^j(k+1)$, for all $j\in\mathcal{N}_i$, as follows:
		\begin{equation}
			{\bm{v}}_{i}^j(k+1) = \eta_i^j\hat{\bm{v}}_i^j{(k)} + \left(1-\eta_i^j\right){\bm{v}}_i^j{(k)}.
			\label{eq:v_hat_up_agl}
		\end{equation}
		\color{black}
		\item Send ${\bm{v}}_{i}^j(k+1)$ to and receive ${\bm{v}}_{j}^i(k+1)$ from neighbors $j\in\mathcal{N}_i$.
		\item Update the dual variables $\bm{\lambda}_i^j(k+1)$, for all $j\in\mathcal{N}_i$, according to
		\begin{equation}
		 \bm{\lambda}_i^j(k+1) = \bm{\lambda}_i^j(k) + {\eta_{i}^j} \left(\bm{v}_{i}^j(k+1) + \bm{v}_{j}^i(k+1)\right).
		\label{eq:dual_l_up_agl}
		\end{equation}
		\item Send $\bm{\lambda}_i^j(k+1)$ to and receive $\bm{\lambda}_j^i(k+1)$ from neighbors $j\in\mathcal{N}_i$.
	\end{enumerate}
\end{alg}
\end{algorithm}

We are in the position to state the proposed distributed approach, which is shown in Algorithm \ref{alg:std}. In step 1 of Algorithm \ref{alg:std}, each agent updates the local decisions $\bm{u}_i(k+1)$ and an auxiliary variable, which is denoted by $\hat{\bm{v}}_i(k)$ and used to update the shared decisions, by solving the decomposed  problem \eqref{eq:q_tilde}, \color{black} where $\tilde{\bm{v}}_j^i=\bm{v}_j^i(k)$. 
Then, the update of $\bm{v}_i(k+1)$ by \eqref{eq:v_hat_up_agl}, where ${\eta_i^j} \in (0,1)$, uses a convex combination of $\hat{\bm{v}}_i(k)$ and the value at the previous iteration $\bm{v}_i(k)$. Meanwhile, the dual variables are updated by \eqref{eq:dual_l_up_agl}, using the step size $\eta_i^j$, for all $j\in\mathcal{N}_i$. 
We will discuss the choice of ${\eta_i^j}$ later in the convergence analysis. 

Now, we consider the time-varying nature of the communication network. 
Based on Assumptions \ref{as:prob_link} and \ref{as:act_agent}, agent $i$ can only exchange information to its neighbor $j \in\mathcal{N}_i$ if both agents are active and the link $\{i,j\}$ is also active. In this regard, for each agent $i$, we denote the set of coupled neighbors with which agent $i$ can exchange information by $\mathcal{A}_i(k) = \{j\in \mathcal{N}_i\cap\mathcal{A}(k):\{i,j\}\in\mathcal{E}^c(k) \}$. In this situation, an active agent $i \in \mathcal{A}(k+1)$ might not have $\bm{v}_j^i(k)$ and $\bm{\lambda}_j^i(k)$ at iteration $k$. Therefore, it needs to track $\bm{v}_j^i(k)$  and $\bm{\lambda}_j^i(k)$. In this regard, this information is captured by the auxiliary variables $\bm{z}_i^j(k)$ and $\bm{\xi}_i^j(k)$, for all $j\in\mathcal{N}_i$, respectively. \color{black}  
The proposed distributed method follows  Algorithm \ref{alg_tv}.
\begin{algorithm}[t]
\begin{alg} Distributed augmented Lagrangian with imperfect communication\hfill
\smallskip
\hrule
\smallskip 	
\noindent	\textbf{Initialization:} For each agent $i\in\mathcal{N}$, $\bm{v}_i(0)=\bm{v}_{i0} \in \mathbb{R}^{|\mathcal{N}_i| n_s}$ and $\bm{\lambda}_i(0) = \bm{\lambda}_{i0}  \in \mathbb{R}^{|\mathcal{N}_i| n_s}$. Moreover, $\bm{z}_i^j(0) = \bm{v}_{j}^i(0)$ and $\bm{\xi}_i^j(0) = \bm{\lambda}_{j}^i (0)$, for all $j\in\mathcal{N}_i$ and $i\in\mathcal{N}$.
	
\noindent	\textbf{Iteration:}
	For each agent $i\in\mathcal{A}(k+1)$, 
	\label{alg_tv}
	\begin{enumerate}
		\item 
		Update $\left(\bm{u}_i{(k+1)},\hat{\bm{v}}_i{(k)}\right)$ according to
		\begin{align}
		&\left(\bm{u}_i{(k+1)},\hat{\bm{v}}_i{(k)}\right) \notag \\
		&= \arg \min_{(\bm{u}_i,\bm{v}_i) \in \mathcal{C}_i} f_i^{\mathrm{p}}(\bm{u}_{i})+f_i^{\mathrm{s}}(\bm{v}_{i}) +\label{eq:primal_uv_up_agl_tv}\\
		&\quad  \sum_{j\in \mathcal{N}_i}\left(   \langle\bm{\lambda}_i^j(k)+\bm{\xi}_i^j(k),\bm{v}_{i}^j\rangle+\|\bm{v}_{i}^j + \bm{z}_{i}^j(k) \|_2^2\right). \notag
		\end{align}
	\item Update ${\bm{v}}_{i}^j(k+1)$, for all $j\in\mathcal{N}_i$, as follows:
	\begin{align}
	&\bm{v}_i^j(k+1) = \begin{cases}
	\eta_i^j\hat{\bm{v}}_i^j{(k)} + \left(1-\eta_i^j\right){\bm{v}}_i^j{(k)},\\ \qquad \quad \ \forall j \in \mathcal{A}_i(k+1),\\
	\bm{v}_i^j(k), \quad \text{otherwise}.
	\end{cases}\label{eq:v_hat_up_agl_tv}
	\end{align}
	\item Send ${\bm{v}}_{i}^j(k+1)$ to and receive ${\bm{v}}_{j}^i(k+1)$ from neighbor $j\in\mathcal{A}_i(k+1)$.
	\item Update the auxiliary and dual variables $\bm{z}_i(k+1)$ and $\bm{\lambda}_i(k+1)$ according to:
	\begin{align}
	&\bm{z}_i^j(k+1) = \begin{cases}
	\bm{v}_{j}^i(k+1),  \quad \forall j \in \mathcal{A}_i(k+1),\\
	\bm{z}_i^j(k),\qquad \quad \text{otherwise},
	\end{cases}\label{eq:z_up_agl_tv}\\
	&\bm{\lambda}_i^j(k+1) 
	= \begin{cases}
	\bm{\lambda}_i^j(k) + {\eta_i^j}\hspace{-2pt}\left(\bm{v}_{i}^j(k+1) + \bm{z}_{i}^j(k+1)\right)\hspace{-2pt},\\
	\qquad \qquad \forall j \in \mathcal{A}_i(k+1),\\
	\bm{\lambda}_i^j(k), \quad \text{otherwise}. 
	\end{cases} \label{eq:dual_l_up_agl_tv}
	\end{align}
	\item Send $\bm{\lambda}_i^j(k+1)$ to and receive $\bm{\lambda}_j^i(k+1)$ from neighbors $j\in\mathcal{A}_i(k+1)$.
	\item Update the auxiliary variable $\bm{\xi}_i(k+1)$ according to
	\begin{equation}
		\bm{\xi}_i^j(k+1) = \begin{cases}
		\bm{\lambda}_{j}^i(k+1),  \quad \forall j \in \mathcal{A}_i(k+1),\\
		\bm{\xi}_i^j(k),\qquad \quad \text{otherwise}.
		\end{cases}\label{eq:xi_up_agl_tv}
	\end{equation}
	\end{enumerate}
	\vspace{-5pt}
	For agent $i\notin\mathcal{A}(k+1)$, $\bm{u}_i{(k+1)}=\bm{u}_i{(k)}$,  $\bm{v}_i{(k+1)}=\bm{v}_i{(k)}$, $\bm{z}_i(k+1)=\bm{z}_i(k)$,  $\bm{\lambda}_i(k+1)=\bm{\lambda}_i(k)$, and $\bm{\xi}_i(k+1)=\bm{\xi}_i(k)$. 
\end{alg}
\end{algorithm}
\begin{remark}
	\label{rem:init}
	In order to initialize the auxiliary variables $\bm{z}_i(0)$ and $\bm{\xi}_i(0)$, either agent $i \in \mathcal{N}$ receives $\bm{v}_j^i(0)$ and $\bm{\lambda}_j^i(0)$ from all neighbors $j\in\mathcal{N}_i$ or it is set such that, for each $i\in\mathcal{N}$, $\bm{v}_i(0)=\bm{z}_i(0)=v_0 \mathds{1}_{|\mathcal{N}_i|n_s}$ and $\bm{\lambda}_i(0)=\bm{\xi}_i(0)=\lambda_0 \mathds{1}_{|\mathcal{N}_i|n_s}$, for any $v_0,\lambda_0 \in \mathbb{R}$. \eod
\end{remark}
\begin{remark}
	The case where the algorithm is performed under perfect communication, as stated in Algorithm \ref{alg:std}, can be considered as a special case of Algorithm \ref{alg_tv} where $\beta_{ij}=1$, for all $\{i,j\}\in\mathcal{E}$ and $\gamma_{i}=1$, for all $i\in\mathcal{N}$.\eod 
\end{remark}

\subsection{Convergence statement}
The convergence of the sequence produced by Algorithm \ref{alg_tv} is stated in Theorem \ref{th:conv}, as follows.
\begin{theorem}
	\label{th:conv}
	Let Assumptions \ref{as:strict_cvx_cost}-\ref{as:act_agent} hold. Furthermore, let the sequence $\{(\bm{u}(k),\bm{v}(k),\hat{\bm{v}}(k),\bm{\lambda}(k))\}$ be generated by Algorithm \ref{alg_tv}. 
	If $\eta_i^j=\eta_j^i=\eta_{ij} \in \left(0,\frac{1}{4}\right)$, for all $j\in\mathcal N_i$ and $i\in\mathcal N$, then, with probability 1,
	\begin{enumerate}[label=\alph*.]
		\item (Feasibility)  $\lim_{k\to\infty} \|\bm{v}_i^j(k)+\bm{v}_j^i(k)\|_2^2 = 0$, for all $j\in\mathcal{N}_i$ and $i\in \mathcal{N}$,
		\item (Primal and dual variable convergence) There exists a saddle point of $L(\bm{u},\bm{v},\bm{\lambda})$  (see \eqref{eq:aug_Lagr}), denoted by $(\bm{u}^{\star},\bm{v}^{\star},\bm{\lambda}^{\star})$, such that $\lim_{k\to\infty} \bm{u}(k)= \bm{u}^{\star}$, $\lim_{k\to\infty} \bm{v}(k)= \bm{v}^{\star}$, and $\lim_{k\to\infty} \bm{\lambda}(k)=\bm{\lambda}^{\star}$. \eod 
	\end{enumerate}
\end{theorem}
\begin{proof}
	See Section \ref{sec:pf_th_conv}.
\end{proof}

Notice that,	if the dual variables $\bm{\lambda}_i(0)$, for all $i\in\mathcal{N}$, are initialized such that  $\bm{\lambda}_{i0}^j=\bm{\lambda}_{j0}^i$, then, we have that $\bm{\lambda}_i^j(k) = \bm{\lambda}_j^i(k)$, for all $k\geq0$, since $\eta_i^j=\eta_j^i=\eta_{ij}$. In this setup, the second round of communication (Step 5) in Algorithm \ref{alg_tv} is not necessary and each agent $i\in\mathcal{A}(k+1)$ can update $\bm{\xi}_i^j(k+1)=\bm{\lambda}_i^j(k+1)$, for all $j\in\mathcal{A}_i(k+1)$, and $\bm{\xi}_i^j(k+1)=\bm{\xi}_i^j(k)$, otherwise.

We also state the convergence rate of Algorithm \ref{alg_tv} in terms of the ergodic average of the primal and auxiliary variables, which are defined, 
for all $i\in\mathcal{N}$ and $k\geq 1$, as follows:
\begin{equation}
\bar{\bm{u}}_i({k})=\sum_{\ell=0}^{{k}-1}\frac{\bm{u}_i(\ell)}{{k}}, \ 
\bar{\bm{v}}_i({k})=\sum_{\ell=0}^{{k}-1}\frac{\bm{v}_i(\ell)}{{k}}, \
\bar{\hat{\bm{v}}}_i({k})=\sum_{\ell=0}^{{k}-1}\frac{\hat{\bm{v}}_i(\ell)}{{k}}, \
\label{eq:erg_ave}
\end{equation}
\begin{theorem} 
	\label{th:rate}
	Let Assumptions \ref{as:strict_cvx_cost}-\ref{as:act_agent} hold. Furthermore, let the sequence $\{(\bm{u}(k),\bm{v}(k),\hat{\bm{v}}(k),\bm{\lambda}(k))\}$ be generated by Algorithm \ref{alg_tv} with $\eta_i^j=\eta_j^i=\eta_{ij} \in \left(0,\frac{1}{4}\right)$, for all $j\in\mathcal N_i$. 
	Then, the ergodic average of the primal variables \eqref{eq:erg_ave} converge to a solution to Problem \eqref{eq:opt} with the convergence rate $\mathcal{O}(\frac{1}{k})$. 
\end{theorem}
\begin{proof}
	See Section \ref{sec:pf_th_rate}.
\end{proof}


\section{Convergence and rate analysis}
\label{sec:conv_an}
This section is devoted to proving Theorems \ref{th:conv} and \ref{th:rate}. Prior to that, we establish some intermediate results that are useful for proving these theorems. \color{black}
\vspace{-10pt}
\subsection{Intermediate Results}
\begin{lemma}
	\label{le:ineq_opt3}
	Let Assumptions \ref{as:strict_cvx_cost}-\ref{as:nonempty_set} hold. Furthermore, let $(\bm{u}_i(k+1),\hat{\bm{v}}_i(k))$ be the attainer of the local optimization in \eqref{eq:primal_uv_up_agl} and $(\bm{u}^{\star},\bm{v}^{\star},\bm{\lambda}^{\star})$ be a saddle point of $L(\bm{u},\bm{v},\bm{\lambda})$ as defined in \eqref{eq:aug_Lagr}. Then, it holds that
	\begin{equation}
	\begin{aligned}
	0  
	&\leq \sum_{i\in\mathcal{N}} \Big(  {-m_i}\|\bm{u}_i(k+1)-\bm{u}_i^{\star}\|_2^2   \Big.\\
	&\quad   + \sum_{j\in \mathcal{N}_i} \langle \bm{\lambda}_i^{j\star}-\bm{\lambda}_i^j(k), \hat{\bm{v}}_i^j(k)+\hat{\bm{v}}_j^i(k)\rangle \\
	&\quad   - \sum_{j\in \mathcal{N}_i} \| \hat{\bm{v}}_i^j(k) + \hat{\bm{v}}_j^i(k)\|_2^2\\
	&\quad  \Big. -2 \sum_{j\in \mathcal{N}_i} \langle {\bm{v}}_j^i(k) -  \hat{\bm{v}}_j^i(k), \hat{\bm{v}}_i^j(k)-\bm{v}_i^{j\star}\rangle\Big).
	\end{aligned}
	\label{eq:ineq_opt3}
	\end{equation}
\end{lemma}
\begin{proof}
		Since $(\bm{u}_i^{\star},\bm{v}_i^{\star}) \in \mathcal{C}_i$, the optimality condition \cite[Theorem 20]{nedic2008} of the local optimization in \eqref{eq:primal_uv_up_agl} yields the following relation:
	\begin{equation}
	\begin{aligned}
	0 &\leq \langle \nabla f_i^{\mathrm{p}}(\bm{u}_i(k+1)), \bm{u}_i^{\star} - \bm{u}_i(k+1)\rangle  \\
	&\quad +  \langle \nabla f_i^{\mathrm{s}}(\hat{\bm{v}}_i(k)), \bm{v}_i^{\star}-\hat{\bm{v}}_i(k)\rangle  \\
	&\quad  + \sum_{j\in \mathcal{N}_i} \langle \bm{\lambda}_i^j(k)+\bm{\lambda}_j^i(k) , \bm{v}_i^{j\star}-\hat{\bm{v}}_i^j(k)\rangle\\
	&\quad  + \sum_{j\in \mathcal{N}_i} 2\langle \hat{\bm{v}}_i^j(k) + {\bm{v}}_j^i(k), \bm{v}_i^{j\star}-\hat{\bm{v}}_i^j(k)\rangle . 		
	\end{aligned}
	\label{eq:opt_c1}
	\end{equation}
	Now, we consider the second inequality in \eqref{eq:saddle_point}, which implies that 
	$(\bm{u}^{\star},\bm{v}^{\star}) =\arg\min_{(\bm{u}_i,\bm{v}_i) \in \mathcal{C}_i,i\in\mathcal{N}} L(\bm{u},\bm{v},\bm{\lambda}^{\star}).$ 
	Based on the optimality condition of this minimization and the fact that $(\bm{u}_i(k+1),\hat{\bm{v}}_i(k)) \in \mathcal{C}_i$, we obtain that 
	\begin{equation}
	\begin{aligned}
	0 &\leq  \sum_{i\in\mathcal{N}} \Big( \langle\nabla f_i^{\mathrm{p}}(\bm{u}_i^{\star}), \bm{u}_i(k+1) - \bm{u}_i^{\star}\rangle  \Big. \\
	& \qquad  
	+ \langle \nabla f_i^{\mathrm{s}}(\bm{v}_i^{\star}), \hat{\bm{v}}_i(k)-\bm{v}_i^{\star}\rangle  \\
	&\qquad \quad  
	+ \sum_{j\in \mathcal{N}_i} \langle \bm{\lambda}_i^{j\star}+\bm{\lambda}_j^{i\star}, \hat{\bm{v}}_i^j(k)-\bm{v}_i^{j\star}\rangle. 
	\end{aligned}
	\label{eq:opt_c2}
	\end{equation}
\color{black}
	By summing up \eqref{eq:opt_c1} over all agents $i\in\mathcal{N}$ and combining with \eqref{eq:opt_c2}, we obtain that
	\begin{align}
	& 0 \leq  \sum_{i\in\mathcal{N}} \Big( \Big. \langle\nabla f_i^{\mathrm{p}}(\bm{u}_i^{\star})-\nabla f_i^{\mathrm{p}}(\bm{u}_i(k+1)), \bm{u}_i(k+1) - \bm{u}_i^{\star}\rangle   \notag\\
	& \qquad \quad + \langle \nabla f_i^{\mathrm{s}}(\bm{v}_i^{\star})-\nabla f_i^{\mathrm{s}}(\hat{\bm{v}}_i(k)), \hat{\bm{v}}_i(k)-\bm{v}_i^{\star}\rangle  \notag\\
	&\qquad \quad  + \sum_{j\in \mathcal{N}_i} \langle \bm{\lambda}_i^{j\star}+\bm{\lambda}_j^{i\star}-\bm{\lambda}_i^j(k)-\bm{\lambda}_j^i(k), \hat{\bm{v}}_i^j(k)-\bm{v}_i^{j\star}\rangle \notag \\
	&\qquad \quad  {-2} \sum_{j\in \mathcal{N}_i} \langle \hat{\bm{v}}_i^j(k) + {\bm{v}}_j^i(k), \hat{\bm{v}}_i^j(k)-\bm{v}_i^{j\star}\rangle\Big.\Big).\label{eq:ineq_opt}
	\end{align}	 
	Applying the convexity and strong convexity relations (cf. Definitions \ref{def:cvx}-\ref{def:strong_cvx} for $f_i^{\mathrm{s}}(\bm{v}_i)$ and $f_i^{\mathrm{p}}(\bm{u}_i)$, for all $i\in\mathcal{N}$, and adding the term $\sum_{i\in\mathcal{N}}\sum_{j\in \mathcal{N}_i}2\langle \hat{\bm{v}}_j^i(k)- \hat{\bm{v}}_j^i(k), \hat{\bm{v}}_i^j(k)-\bm{v}_i^{j\star}\rangle=0$  to \eqref{eq:ineq_opt}, it follows that
	\begin{align}
	0  
	& \leq \sum_{i\in\mathcal{N}} \Big(  {-m_i}\|\bm{u}_i(k+1)-\bm{u}_i^{\star}\|_2^2  \Big. \notag\\
	&\quad   + \sum_{j\in \mathcal{N}_i} \langle \bm{\lambda}_i^{j\star}+\bm{\lambda}_j^{i\star}-\bm{\lambda}_i^j(k)-\bm{\lambda}_j^i(k), \hat{\bm{v}}_i^j(k)-\bm{v}_i^{j\star}\rangle \notag\\
	&\quad   {-2} \sum_{j\in \mathcal{N}_i} \langle \hat{\bm{v}}_i^j(k) + \hat{\bm{v}}_j^i(k), \hat{\bm{v}}_i^j(k)-\bm{v}_i^{j\star}\rangle \notag\\
	&\quad \Big.  {-2} \sum_{j\in \mathcal{N}_i} \langle {\bm{v}}_j^i(k) -  \hat{\bm{v}}_j^i(k), \hat{\bm{v}}_i^j(k)-\bm{v}_i^{j\star}\rangle\Big).
	\label{eq:ineq_opt2}
	\end{align}
	
	Now, we consider the second term on the right-hand side of the inequality, \color{black}  i.e., $\sum_{i\in\mathcal{N}}\sum_{j\in \mathcal{N}_i} \langle \bm{\lambda}_i^{j\star}+\bm{\lambda}_j^{i\star}-\bm{\lambda}_i^j(k)-\bm{\lambda}_j^i(k), \hat{\bm{v}}_i^j(k)-\bm{v}_i^{j\star}\rangle.$ By considering the summation over all links and since at each link there exist two inner products associated to both agents coupled by that link, that term is equivalent to 
	$\sum_{i\in\mathcal{N}}\sum_{j\in \mathcal{N}_i} \langle \bm{\lambda}_i^{j\star}-\bm{\lambda}_i^j(k), \hat{\bm{v}}_i^j(k)+\hat{\bm{v}}_j^i(k)\rangle.$ 
	\color{black} 
	Similarly, the third term on the right-hand side of the inequality \eqref{eq:ineq_opt2}, i.e., $2\sum_{i\in\mathcal{N}}\sum_{j\in \mathcal{N}_i}\langle \hat{\bm{v}}_i^j(k) + \hat{\bm{v}}_j^i(k), \hat{\bm{v}}_i^j(k)-\bm{v}_i^{j\star}\rangle$ is equivalent to 
	$\sum_{\{i,j\}\in \mathcal{E}} 2\langle \hat{\bm{v}}_i^j(k) + \hat{\bm{v}}_j^i(k),\hat{\bm{v}}_i^j(k)+\hat{\bm{v}}_j^i(k)\rangle
	= \sum_{i\in\mathcal{N}}\sum_{j\in \mathcal{N}_i} \| \hat{\bm{v}}_i^j(k) + \hat{\bm{v}}_j^i(k)\|_2^2$. 
	Thus,  we obtain the desired inequality \eqref{eq:ineq_opt3}.
\end{proof}

Now, 
we define the auxiliary variables, $\tilde{\bm{\lambda}}_i^j(k)$, for all $j\in\mathcal{N}_i$ and $i\in\mathcal{N}$, as follows:  
\begin{equation}
\tilde{\bm{\lambda}}_i^j(k) = \bm{\lambda}_i^j(k) + (1-\eta_i^j)(\bm{v}_i^j(k)+\bm{v}_j^i(k)),
\label{eq:lambda_bar}
\end{equation}
and obtain a useful estimate in Lemma \ref{le:ineq_opt}. \color{black}
\begin{lemma}
	\label{le:ineq_opt}
	Let Assumptions \ref{as:strict_cvx_cost}-\ref{as:nonempty_set} hold. Furthermore, let $(\bm{u}_i(k+1),\hat{\bm{v}}_i(k))$ be the attainer of the local optimization in \eqref{eq:primal_uv_up_agl}, $(\bm{u}^{\star},\bm{v}^{\star},\bm{\lambda}^{\star})$ be a saddle point of $L(\bm{u},\bm{v},\bm{\lambda})$  as defined in \eqref{eq:aug_Lagr}, and $\tilde{\bm{\lambda}}_i^j(k)$ be defined as in \eqref{eq:lambda_bar}. Then, it holds that
	\begin{equation}
	\begin{aligned}
	&\sum_{i\in\mathcal{N}}  \sum_{j\in \mathcal{N}_i} \langle \tilde{\bm{\lambda}}_i^j(k)-\bm{\lambda}_i^{j\star}, \hat{\bm{v}}_i^j(k)+\hat{\bm{v}}_j^i(k)\rangle\\
	&\quad + 2\sum_{i\in\mathcal{N}} \sum_{j\in \mathcal{N}_i} \langle \hat{\bm{v}}_i^j(k)-\bm{v}_i^j(k),{\bm{v}}_i^j(k)-\bm{v}_i^{j\star}\rangle  \\
	&\leq \sum_{i\in\mathcal{N}} \Big(  {-m_i}\|\bm{u}_i(k+1)-\bm{u}_i^{\star}\|_2^2  \Big. \\
	& \quad - \sum_{j\in \mathcal{N}_i} \frac{3}{2}\|\hat{\bm{v}}_i^j(k)-\bm{v}_i^j(k) \|_2^2  \\
	& \quad\Big. - \sum_{j\in \mathcal{N}_i} \frac{\eta_i^j+\eta_j^i-(\eta_i^j+\eta_j^i)^2}{2}\|\hat{\bm{v}}_i^j(k)+\hat{\bm{v}}_j^i(k)\|_2^2\Big).
	\end{aligned}
	\label{eq:ineq_opt7}
	\end{equation}
\end{lemma}
\begin{proof}
	We use Lemma \ref{le:ineq_opt3}, where we rearrange \eqref{eq:ineq_opt3} and add the term $2\sum_{i\in\mathcal{N}} \sum_{j\in \mathcal{N}_i} \langle \hat{\bm{v}}_i^j(k)-\bm{v}_i^j(k),\hat{\bm{v}}_i^j(k)-\bm{v}_i^{j\star}\rangle$ on both sides of the inequality. We obtain that
	\begin{equation}
	\begin{aligned}
	&\sum_{i\in\mathcal{N}}  \sum_{j\in \mathcal{N}_i} \langle \bm{\lambda}_i^j(k)-\bm{\lambda}_i^{j\star}, \hat{\bm{v}}_i^j(k)+\hat{\bm{v}}_j^i(k)\rangle\\
	&\quad + 2\sum_{i\in\mathcal{N}} \sum_{j\in \mathcal{N}_i} \langle \hat{\bm{v}}_i^j(k)-{\bm{v}}_i^j(k),\hat{\bm{v}}_i^j(k)-\bm{v}_i^{j\star}\rangle  \\
	&\leq \sum_{i\in\mathcal{N}} \Big(  {-m_i}\|\bm{u}_i(k+1)-\bm{u}_i^{\star}\|_2^2  \Big. \\
	&\quad   - \sum_{j\in \mathcal{N}_i} \| \hat{\bm{v}}_i^j(k) + \hat{\bm{v}}_j^i(k)\|_2^2\\
	&\quad   +2 \sum_{j\in \mathcal{N}_i} \langle \hat{\bm{v}}_j^i(k)-{\bm{v}}_j^i(k) , \hat{\bm{v}}_i^j(k)-\bm{v}_i^{j\star}\rangle\\
	&\quad \Big.  +2 \sum_{j\in \mathcal{N}_i} \langle \hat{\bm{v}}_i^j(k)-{\bm{v}}_i^j(k) , \hat{\bm{v}}_i^j(k)-\bm{v}_i^{j\star}\rangle\Big).
	\end{aligned}
	\label{eq:ineq_opt8}
	\end{equation}
	The term in the second summation on the left-hand side of the inequality can be expressed as follows:
$		\langle \hat{\bm{v}}_i^j(k)-{\bm{v}}_i^j(k),\hat{\bm{v}}_i^j(k)-\bm{v}_i^{j\star}\rangle
		= \langle \hat{\bm{v}}_i^j(k)-{\bm{v}}_i^j(k),{\bm{v}}_i^j(k)-\bm{v}_i^{j\star}\rangle + \|\hat{\bm{v}}_i^j(k)-\bm{v}_i^j(k) \|_2^2$. 
	\color{black}
	
	Moreover, for the last two terms on the right-hand side of \eqref{eq:ineq_opt8}, we have 
$	2\sum_{i\in\mathcal{N}}\sum_{j\in \mathcal{N}_i} \Big( \langle \hat{\bm{v}}_j^i(k) -  {\bm{v}}_j^i(k), \hat{\bm{v}}_i^j(k)-\bm{v}_i^{j\star}\rangle
	+\langle \hat{\bm{v}}_i^j(k)-{\bm{v}}_i^j(k),\hat{\bm{v}}_i^j(k)-\bm{v}_i^{j\star}\rangle  \Big)$,
	which is equivalent to
	$2 \sum_{\{i,j\}\in \mathcal{E}} \Big(\Big. \langle \hat{\bm{v}}_j^i(k) -  {\bm{v}}_j^i(k) + \hat{\bm{v}}_i^j(k)-{\bm{v}}_i^j(k), \hat{\bm{v}}_i^j(k)+\hat{\bm{v}}_i^{j\star}\rangle\\
	+ \langle \hat{\bm{v}}_j^i(k) -  {\bm{v}}_j^i(k) + \hat{\bm{v}}_i^j(k)-{\bm{v}}_i^j(k), \hat{\bm{v}}_j^i(k)+\hat{\bm{v}}_j^{i\star}\rangle\Big.\Big)=\\
	 \sum_{i\in\mathcal{N}}\sum_{j\in \mathcal{N}_i}\langle   \hat{\bm{v}}_i^j(k)+  \hat{\bm{v}}_j^i(k)- {\bm{v}}_i^j(k)-\bm{v}_j^i(k), \hat{\bm{v}}_i^j(k)+\hat{\bm{v}}_j^i(k)\rangle$. 
	\color{black}
	Thus, applying the two preceding relations to \eqref{eq:ineq_opt8}, we have that
	\begin{align}
	&\sum_{i\in\mathcal{N}}  \sum_{j\in \mathcal{N}_i} \langle \bm{\lambda}_i^j(k)-\bm{\lambda}_i^{j\star}, \hat{\bm{v}}_i^j(k)+\hat{\bm{v}}_j^i(k)\rangle \notag\\
	&\quad + 2\sum_{i\in\mathcal{N}} \sum_{j\in \mathcal{N}_i} \langle \hat{\bm{v}}_i^j(k)-{\bm{v}}_i^j(k),{\bm{v}}_i^j(k)-\bm{v}_i^{j\star}\rangle  \leq \notag\\
	&\sum_{i\in\mathcal{N}} \Big( \Big.  - \sum_{j\in \mathcal{N}_i} \left(\| \hat{\bm{v}}_i^j(k) + \hat{\bm{v}}_j^i(k)\|_2^2+ 2\|\hat{\bm{v}}_i^j(k)-\bm{v}_i^j(k) \|_2^2\right)   \notag\\
	&\quad + \sum_{j\in \mathcal{N}_i} \langle   \hat{\bm{v}}_i^j(k)+  \hat{\bm{v}}_j^i(k)- {\bm{v}}_i^j(k)-\bm{v}_j^i(k), \hat{\bm{v}}_i^j(k)+\hat{\bm{v}}_j^i(k)\rangle \notag\\ 
	&\quad  -m_i\|\bm{u}_i(k+1)-\bm{u}_i^{\star}\|_2^2\Big.\Big). \label{eq:ineq_opt4}
	\end{align}
	Furthermore, adding the term $\sum_{i\in\mathcal{N}} \sum_{j\in \mathcal{N}_i} (1-\eta_i^j)\langle\bm{v}_i^j(k)+\bm{v}_j^i(k),\hat{\bm{v}}_i^j(k)+\hat{\bm{v}}_j^i(k)  \rangle$ to both sides of the inequality in \eqref{eq:ineq_opt4} and recalling the definition of $\tilde{\bm{\lambda}}_i^j(k)$ in \eqref{eq:lambda_bar}, it follows that
	\begin{align}
	&\sum_{i\in\mathcal{N}}  \sum_{j\in \mathcal{N}_i} \langle \tilde{\bm{\lambda}}_i^j(k)-\bm{\lambda}_i^{j\star}, \hat{\bm{v}}_i^j(k)+\hat{\bm{v}}_j^i(k)\rangle \notag\\
	&\quad + 2\sum_{i\in\mathcal{N}} \sum_{j\in \mathcal{N}_i} \langle \hat{\bm{v}}_i^j(k)-\bm{v}_i^j(k),{\bm{v}}_i^j(k)-\bm{v}_i^{j\star}\rangle  \notag\\
	&\leq \sum_{i\in\mathcal{N}} \Big(  {-m_i}\|\bm{u}_i(k+1)-\bm{u}_i^{\star}\|_2^2  \Big. \notag\\
	&\quad   - \sum_{j\in \mathcal{N}_i} \left(\| \hat{\bm{v}}_i^j(k) + \hat{\bm{v}}_j^i(k)\|_2^2+ 2\|\hat{\bm{v}}_i^j(k)-\bm{v}_i^j(k) \|_2^2\right) \notag\\ 
	&\quad   + \sum_{j\in \mathcal{N}_i} \langle \hat{\bm{v}}_i^j(k)+  \hat{\bm{v}}_j^i(k)- {\bm{v}}_i^j(k)-\bm{v}_j^i(k), \hat{\bm{v}}_i^j(k)+\hat{\bm{v}}_j^i(k)\rangle \notag\\
	&\quad  \Big. + \sum_{j\in \mathcal{N}_i} (1-\eta_i^j)\langle\bm{v}_i^j(k)+\bm{v}_j^i(k),\hat{\bm{v}}_i^j(k)+\hat{\bm{v}}_j^i(k)  \rangle\Big).\label{eq:ineq_opt5}
	\end{align}
	Now, consider the last two terms on the right-hand side of \eqref{eq:ineq_opt5}. By adding them with  $\sum_{i\in\mathcal{N}}\sum_{j\in \mathcal{N}_i}(1-\eta_i^j)\left(\|\hat{\bm{v}}_i^j(k)+\hat{\bm{v}}_j^i(k)\|_2^2 - \|\hat{\bm{v}}_i^j(k)+\hat{\bm{v}}_j^i(k)\|_2^2\right)=0 $, we obtain 
	\begin{equation*}
		\begin{aligned}
		&\sum_{i\in\mathcal{N}}\sum_{j\in \mathcal{N}_i} \Big( (1-\eta_i^j)\langle\bm{v}_i^j(k)+\bm{v}_j^i(k),\hat{\bm{v}}_i^j(k)+\hat{\bm{v}}_j^i(k)\rangle \Big.\\
		&\quad + \langle \hat{\bm{v}}_i^j(k)+  \hat{\bm{v}}_j^i(k)- {\bm{v}}_i^j(k)-\bm{v}_j^i(k), \hat{\bm{v}}_i^j(k)+\hat{\bm{v}}_j^i(k)\rangle\Big.\Big) \\
		& = \sum_{i\in\mathcal{N}}\sum_{j\in \mathcal{N}_i} \Big( (1-\eta_i^j) \|\hat{\bm{v}}_i^j(k)+\hat{\bm{v}}_j^i(k)\|_2^2+\Big.\\
		& \quad \quad \Big.\eta_i^j \langle \hat{\bm{v}}_i^j(k)+  \hat{\bm{v}}_j^i(k)- {\bm{v}}_i^j(k)-\bm{v}_j^i(k), \hat{\bm{v}}_i^j(k)+\hat{\bm{v}}_j^i(k)\rangle\Big).
		\end{aligned}
	\end{equation*}
	\color{black}
	Therefore, \eqref{eq:ineq_opt5} becomes
	\begin{equation}
	\begin{aligned}
	&\sum_{i\in\mathcal{N}}  \sum_{j\in \mathcal{N}_i} \langle \tilde{\bm{\lambda}}_i^j(k)-\bm{\lambda}_i^{j\star}, \hat{\bm{v}}_i^j(k)+\hat{\bm{v}}_j^i(k)\rangle\\
	&\quad + 2\sum_{i\in\mathcal{N}} \sum_{j\in \mathcal{N}_i} \langle \hat{\bm{v}}_i^j(k)-\bm{v}_i^j(k),{\bm{v}}_i^j(k)-\bm{v}_i^{j\star}\rangle \leq \\
	&\sum_{i\in\mathcal{N}} \Big(  {-m_i}\|\bm{u}_i(k+1)-\bm{u}_i^{\star}\|_2^2  \Big. \\
	& + \sum_{j\in \mathcal{N}_i} \eta_i^j\langle \hat{\bm{v}}_i^j(k)+  \hat{\bm{v}}_j^i(k)- {\bm{v}}_i^j(k)-\bm{v}_j^i(k), \hat{\bm{v}}_i^j(k)+\hat{\bm{v}}_j^i(k)\rangle\\
	&   \Big.- \sum_{j\in \mathcal{N}_i} \left(\eta_i^j\| \hat{\bm{v}}_i^j(k) + \hat{\bm{v}}_j^i(k)\|_2^2+ 2\|\bm{v}_i^j(k)-\hat{\bm{v}}_i^j(k) \|_2^2\right)\Big).
	\end{aligned}
	\label{eq:ineq_opt6}
	\end{equation}
	Now, we compute an upper-bound for the term $$\sum_{i\in\mathcal{N}} \sum_{j\in \mathcal{N}_i}\eta_i^j\langle \hat{\bm{v}}_i^j(k)+  \hat{\bm{v}}_j^i(k)- {\bm{v}}_i^j(k)-\bm{v}_j^i(k), \hat{\bm{v}}_i^j(k)+\hat{\bm{v}}_j^i(k)\rangle,$$ on the right-hand side of the inequality in \eqref{eq:ineq_opt6}. To that end, this term can be written as
	\begin{align*}
		\sum_{i\in\mathcal{N}} \sum_{j\in \mathcal{N}_i}-(\eta_i^j+\eta_j^i)\langle {\bm{v}}_i^j(k)  - \hat{\bm{v}}_i^j(k), \hat{\bm{v}}_i^j(k)+\hat{\bm{v}}_j^i(k)\rangle.
	\end{align*}
	Using the fact that, for any $\xi \in \mathbb{R}, a\in\mathbb{R}^n, b\in\mathbb{R}^n$, $\|a+\xi b\|_2^2=\|a\|_2^2 + \xi^2\|b\|_2^2+2\xi\langle a,b \rangle \Rightarrow -\xi \langle a,b \rangle \leq \frac{1}{2} \left(\|a\|_2^2 + \xi^2\|b\|_2^2\right) $, we obtain an upper-bound of the term inside the summation, i.e.,
	 $-(\eta_i^j+\eta_j^i)\langle {\bm{v}}_i^j(k)  - \hat{\bm{v}}_i^j(k), \hat{\bm{v}}_i^j(k)+\hat{\bm{v}}_j^i(k)\rangle 
	  \leq  \frac{1}{2}\left(\|{\bm{v}}_i^j(k)  - \hat{\bm{v}}_i^j(k)\|_2^2 + (\eta_i^j+\eta_j^i)^2\|\hat{\bm{v}}_i^j(k)+\hat{\bm{v}}_j^i(k)\|_2^2 \right)$. \color{black}
	 Therefore, using the above upper-bound and the fact that 
	 $\sum_{i\in\mathcal{N}}\sum_{j\in \mathcal{N}_i}\eta_i^j\|\hat{\bm{v}}_i^j(k)+\hat{\bm{v}}_j^i(k)\|_2^2 = \sum_{i\in\mathcal{N}}\sum_{j\in \mathcal{N}_i}\frac{\eta_i^j+\eta_j^i}{2}\|\hat{\bm{v}}_i^j(k)+\hat{\bm{v}}_j^i(k)\|_2^2,
	 $ 
	 the desired inequality \eqref{eq:ineq_opt7} follows.
\end{proof}

As the next building block to show the convergence result, we define a Lyapunov function, denoted by $V(k)$. For any given saddle point of $L(\bm{u},\bm{v},\bm{\lambda})$ (see \eqref{eq:aug_Lagr}), denoted by $(\bm{u}^{\star},\bm{v}^{\star},\bm{\lambda}^{\star})$, we can construct $V(k)$ as follows:
\begin{equation}
\small
V(k) = \|\bm{v}(k)-\bm{v}^{\star}\|_{H}^2+\frac{1}{2}\|\tilde{\bm{\lambda}}(k)-\bm{\lambda}^{\star}\|_H^2,
\label{eq:Vk}
\end{equation}
where $\tilde{\bm{\lambda}}=[\tilde{\bm{\lambda}}_i(k)]_{i\in\mathcal{N}}$, $\tilde{\bm{\lambda}}_i(k)=[\tilde{\bm{\lambda}}_i^j(k)]_{j\in\mathcal{N}_i}$, $\tilde{\bm{\lambda}}_i^j(k)$ is defined in \eqref{eq:lambda_bar}, 
$H = \operatorname{blkdiag}(\{H_i\}_{i\in\mathcal{N}})$ and
$H_i = \operatorname{blkdiag}(\{(\eta_i^j)^{-1}I_{n_s}\}_{j\in\mathcal{N}_i})$,  for all $i\in\mathcal{N}$. Now, we show that $\{V(k)\}$ is non-increasing under Algorithm \ref{alg:std} and obtain an estimate that will be used in the main theorems. 

\begin{lemma}
	\label{le:Vk_ub}
	Let Assumptions \ref{as:strict_cvx_cost}-\ref{as:nonempty_set} hold. Furthermore, let the sequence $\{\bm{u}(k),\bm{v}(k),\hat{\bm{v}}(k),\bm{\lambda}(k)\}$ be generated by Algorithm \ref{alg:std}, $(\bm{u}^{\star},\bm{v}^{\star},\bm{\lambda}^{\star})$ be a saddle point of $L(\bm{u},\bm{v},\bm{\lambda})$  as defined in \eqref{eq:aug_Lagr}, and $V(k)$ be defined in \eqref{eq:Vk}. 	
	If $\eta_i^j=\eta_j^i=\eta_{ij} \in \left(0,\frac{1}{4}\right)$, then $\{V(k)\}$ is a monotonically non-increasing sequence and the following inequality holds:
	\begin{equation}
	\begin{aligned}
		&V(k+1)-V(k) \\
		& \leq- \sum_{i\in\mathcal{N}} m_i\|\bm{u}_i(k+1)-\bm{u}_i^{\star}\|_2^2   \\
		& \quad - \sum_{i\in\mathcal{N}}  \sum_{j\in \mathcal{N}_i} \left(\frac{3}{2}-\eta_{ij}\right)\|\hat{\bm{v}}_i^j(k)-\bm{v}_i^j(k) \|_2^2  \\
		& \quad - \sum_{i\in\mathcal{N}} \sum_{j\in \mathcal{N}_i} \frac{\eta_{ij}-(2\eta_{ij})^2}{2}\|\hat{\bm{v}}_i^j(k)+\hat{\bm{v}}_j^i(k)\|_2^2.
	\end{aligned}
	\label{eq:Vk_ub}
	\end{equation}
\end{lemma}
\begin{proof}
		Firstly, notice that $\tilde{\bm{\lambda}}_i^j(k+1)$ can be expressed as 
	$\tilde{\bm{\lambda}}_i^j(k+1) 
	=\tilde{\bm{\lambda}}_i^j(k) + \eta_i^j(\hat{\bm{v}}_i^j(k)+\hat{\bm{v}}_j^i(k)).$ 
	\color{black} Thus, 
	we have that
	\begin{align}
	&\|\tilde{\bm{\lambda}}_i^j(k+1)-\bm{\lambda}_i^{j\star}\|_2^2 \notag\\
	&=\|\tilde{\bm{\lambda}}_i^j(k)-\bm{\lambda}_i^{j\star}\|_2^2 + 2\eta_i^j\langle \tilde{\bm{\lambda}}_i^j(k)-\bm{\lambda}_i^{j\star}, \hat{\bm{v}}_i^j(k)+\hat{\bm{v}}_j^i(k) \rangle \notag\\
	&\quad + \|\eta_i^j(\hat{\bm{v}}_i^j(k)+\hat{\bm{v}}_j^i(k))\|_2^2. \label{eq:Vk_p2}
	\end{align}
	Moreover, we also have that
	\begin{equation}
	\begin{aligned}
	\|\bm{v}_i^j(k+1)-\bm{v}_i^{j\star}\|_2^2
	&=\|\bm{v}_i^j(k)-\bm{v}_i^{j\star}\|_2^2 + \|\eta_i^j(\hat{\bm{v}}_i^j -\bm{v}_i^j(k) )\|_2^2\\
	&\quad + 2\eta_i^j \langle \hat{\bm{v}}_i^j -\bm{v}_i^j(k), \bm{v}_i^j(k)-\bm{v}_i^{j\star} \rangle. 
	\end{aligned}
	\label{eq:Vk_p1}
	\end{equation}
	By using the expression of $V(k+1)$ from \eqref{eq:Vk_p2}-\eqref{eq:Vk_p1} and using the inequality in \eqref{eq:ineq_opt7}, we obtain that:
	\begin{equation*}
	\begin{aligned}
	&V(k+1)-V(k)  \\
	&\leq \|\hat{\bm{v}}(k) -\bm{v}(k)\|_{H^{-1}}^2+\sum_{i\in\mathcal{N}}\sum_{j\in \mathcal{N}_i}\frac{\eta_i^j}{2}\|(\hat{\bm{v}}_i^j(k)+\hat{\bm{v}}_j^i(k))\|_2^2\\
	& \quad - \frac{3}{2}\|\hat{\bm{v}}(k)-\bm{v}(k) \|_2^2 - \sum_{i\in\mathcal{N}} m_i\|\bm{u}_i(k+1)-\bm{u}_i^{\star}\|_2^2   \\
	& \quad - \sum_{i\in\mathcal{N}} \sum_{j\in \mathcal{N}_i} \frac{2\eta_{i}^j-(2\eta_{i}^j)^2}{2}\|\hat{\bm{v}}_i^j(k)+\hat{\bm{v}}_j^i(k)\|_2^2.\\
	\end{aligned}
	\end{equation*}
	Thus, the inequality \eqref{eq:Vk_ub} follows and $V(k)$ is monotonically non-increasing 
	if $\eta_i^j = \eta_j^i=\eta_{ij} \in (0,\frac{1}{4})$.
\end{proof}

\medskip
The function $V(k)$ is used to construct a Lyapunov function for Algorithm \ref{alg_tv}. Moreover, the estimate obtained in Lemma \ref{le:Vk_ub} will also be used to obtain the result in Lemma \ref{le:V_tilde}. Therefore, now consider the function $\tilde{V}(k)$, defined as follows. For any saddle point  of $L(\bm{u},\bm{v},\bm{\lambda})$ (see \eqref{eq:aug_Lagr}), denoted by $(\bm{u}^{\star},\bm{v}^{\star},\bm{\lambda}^{\star})$, we have that
\begin{equation}
\begin{aligned}
\tilde{V}(k) &= \|\bm{v}(k)-\bm{v}^{\star}\|_{\tilde{H}}^2+\frac{1}{2}\|{\bm{\nu}}(k)-\bm{\lambda}^{\star}\|_{\tilde{H}}^2,
\end{aligned}
\label{eq:Vk_tilde}
\end{equation}
where $\bm{\nu}(k)=[\bm{\nu}_i(k)]_{i\in\mathcal{N}}$, $\bm{\nu}_i(k) = [\bm{\nu}_i^j(k)]_{j\in\mathcal{N}_i}$, 
$${\bm{\nu}}_i^j(k) = \bm{\lambda}_i^j(k) + (1-\eta_i^j)(\bm{v}_i^j(k)+\bm{z}_i^j(k)),$$
for all $j\in\mathcal{N}_i$ and $i\in\mathcal{N}$, 
$\tilde{H} = \operatorname{blkdiag}(\{\tilde{H}_i\}_{i\in\mathcal{N}})$, and 
$\tilde{H}_i = \operatorname{blkdiag}(\{(\alpha_{ij}\eta_i^j)^{-1}I_{n_s}\}_{j\in\mathcal{N}_i})$, for all $i\in\mathcal{N}$, where $\alpha_{ij}=\beta_{ij}\gamma_{i}\gamma_{j} \in (0,1]$. Recall that $\beta_{ij}$ and $\gamma_{i}$ are the probability of link $\{i,j\}$ being active and agent $i$ being active, respectively. 
\begin{lemma}
	\label{le:V_tilde}
	Let Assumptions \ref{as:strict_cvx_cost}-\ref{as:act_agent} hold. Furthermore, let the sequence $\{\bm{u}(k),\bm{v}(k),\hat{\bm{v}}(k),\bm{\lambda}(k)\}$ be generated by Algorithm \ref{alg_tv}, $(\bm{u}^{\star},\bm{v}^{\star},\bm{\lambda}^{\star})$ be a saddle point of $L(\bm{u},\bm{v},\bm{\lambda})$  as defined in \eqref{eq:aug_Lagr}, and	$\tilde{V}(k)$ be defined as in \eqref{eq:Vk_tilde}. 	 
	If $\eta_i^j=\eta_j^i=\eta_{ij} \in \left(0,\frac{1}{4}\right)$, for all $j\in\mathcal N_i$ and $i\in\mathcal N$, then the sequence $\{\tilde{V}(k)\}$ is a non-negative supermartingale and it holds with probability 1 that
	\begin{equation}
	\begin{aligned}
	&\mathbb{E}\left(\tilde{V}(k+1)|\mathcal{F}(k)\right)-\tilde{V}(k) \\
	& \leq- \sum_{i\in\mathcal{N}} m_i\|\bm{u}_i(k+1)-\bm{u}_i^{\star}\|_2^2   \\
	& \quad - \sum_{i\in\mathcal{N}}  \sum_{j\in \mathcal{N}_i} \left(\frac{3}{2}-\eta_{ij}\right)\|\hat{\bm{v}}_i^j(k)-\bm{v}_i^j(k) \|_2^2  \\
	& \quad - \sum_{i\in\mathcal{N}} \sum_{j\in \mathcal{N}_i} \frac{\eta_{ij}-(2\eta_{ij})^2}{2}\|\hat{\bm{v}}_i^j(k)+\hat{\bm{v}}_j^i(k)\|_2^2 \leq 0.
	\end{aligned}
	\label{eq:ineq_supermart}
	\end{equation} 
\end{lemma}
\begin{proof}
	Since $\tilde{V}(k)$ is a sum of norms and $\eta_{i}^j$ and $\alpha_{ij}$ are positive, the sequence $\{\tilde{V}(k)\}$ is clearly non-negative. Denote by $\mathcal{F}({k})$  the filtration up to and including the iteration $k$, i.e., $\mathcal{F}({k})=\{\mathcal{A}(\ell),\mathcal{E}^c(\ell),\bm{u}(\ell),\bm{v}(\ell),\bm{\lambda}(\ell),\bm{z}(\ell),\bm{\xi}(\ell),\ \ell=0,1,\dots,k\}$. 
	Now, we show that the conditional expectation of the sequence with respect to $\mathcal{F}({k})$ is always non-increasing. Based on Assumptions \ref{as:prob_link} and \ref{as:act_agent}, \textcolor{black}{a proper initialization in Algorithm \ref{alg_tv}}, and the update rules \eqref{eq:v_hat_up_agl_tv}, \eqref{eq:z_up_agl_tv}, and \eqref{eq:dual_l_up_agl_tv}, the variables  $\bm{v}_i^j(k+1)$,  $\bm{z}_i^j(k+1)=\bm{v}_j^i(k+1)$, and $\bm{\lambda}_i^j(k+1)$, for each $j\in\mathcal{N}_i$, are only updated when agents $i$ and $j$ are active and link $\{i,j\}$ is active. Therefore, we can denote the probability of $\bm{v}_i^j(k+1)$, $\bm{z}_i^j(k+1)=\bm{v}_j^i(k+1)$, and $\bm{\lambda}_i^j(k+1)$ being updated by $\alpha_{ij}=\beta_{ij}\gamma_{i}\gamma_{j} \in (0,1]$, whereas, with probability $1-\alpha_{ij}$, they are not updated and the values remain the same as $\bm{v}_i^j(k)$, $\bm{z}_i^j(k)=\bm{v}_j^i(k)$, and $\bm{\lambda}_i^j(k)$. Thus, we also observe that ${\bm{\nu}}_i^j(k+1) = {\tilde{\bm{\lambda}}}_i^j(k+1) = {\bm{\lambda}}_i^j(k+1) + (1-\eta_{i}^j)({\bm{v}}_i^j(k+1)+{\bm{v}}_j^i(k+1))$ with probability $\alpha_{ij}$ or  the value ${\bm{\nu}}_i^j(k)={\tilde{\bm{\lambda}}}_i^j(k)$ is kept with with probability $1-\alpha_{ij}$. Hence, we obtain, with probability 1, that 
	\begin{equation*}
	\begin{aligned}
	&\mathbb{E}\left(\tilde{V}(k+1)|\mathcal{F}(k)\right)-\tilde{V}(k)= -\tilde{V}(k)   \\
	&\quad +\mathbb{E}\left(\|\bm{v}(k+1)-\bm{v}^{\star}\|_{\tilde{H}}^2+\frac{1}{2}\|{\bm{\nu}}(k+1)-\bm{\lambda}^{\star}\|_{\tilde{H}}^2\Bigg|\mathcal{F}(k)\right) \\
	&= \|{\bm{v}}(k+1)-\bm{v}^{\star}\|_H^2 -   \|{\bm{v}}(k)-\bm{v}^{\star}\|_H^2\\
	&\quad +\frac{1}{2}\|{\tilde{\bm{\lambda}}}(k+1)-\bm{\lambda}^{\star}\|_H^2-\frac{1}{2}\|{\tilde{\bm{\lambda}}}(k)-\bm{\lambda}^{\star}\|_H^2.\\
	\end{aligned}%
	\end{equation*}
	\color{black}
Notice that since the scalings of the remaining quadratic terms do not involve $\alpha_{ij}$, we can use the weighted vector norm induced by $H$. Based on the definition of $V(k)$ given in \eqref{eq:Vk}, we obtain with probability 1 that
	\begin{equation*}
	\begin{aligned}
	&\mathbb{E}\left(\tilde{V}(k+1)|\mathcal{F}(k)\right)-\tilde{V}(k)=V(k+1)-V(k).
		\end{aligned}
	\end{equation*}
	  Therefore, by applying \eqref{eq:Vk_ub} to this relation, the desired relations in \eqref{eq:ineq_supermart} follow, with probability 1,  
	when $\eta_i^j = \eta_j^i=\eta_{ij} \in \left(0,\frac{1}{4}\right)$. Thus, \eqref{eq:ineq_supermart} also shows that the sequence $\{\tilde{V}(k)\}$ is non-negative supermartingale. 
\end{proof}

\subsection{Proof of Theorem \ref{th:conv}}
\label{sec:pf_th_conv}
	Now, we are ready to prove Theorem \ref{th:conv}. Recall the function $\tilde{V}(k)$ defined in \eqref{eq:Vk_tilde} and the inequality \eqref{eq:ineq_supermart} in Lemma \ref{le:V_tilde}. Rearranging and iterating \eqref{eq:ineq_supermart}, for $\ell=0,\dots,{k}$, and taking the total expectation, we have that 
	\begin{equation*}
	\begin{aligned}
	& \sum_{\ell = 0}^{k}\sum_{i\in\mathcal{N}} \mathbb{E} \left(m_i\|\bm{u}_i(\ell+1)-\bm{u}_i^{\star}\|_2^2\right)   \\
	& + \sum_{\ell = 0}^{k}\sum_{i\in\mathcal{N}}  \sum_{j\in \mathcal{N}_i} \left(\frac{3}{2}-\eta_{ij}\right)\mathbb{E}\left(\|\hat{\bm{v}}_i^j(\ell)-\bm{v}_i^j(\ell) \|_2^2\right)  \\
	& + \sum_{\ell = 0}^{k}\sum_{i\in\mathcal{N}} \sum_{j\in \mathcal{N}_i} \frac{\eta_{ij}-(2\eta_{ij})^2}{2}\mathbb{E}\left(\|\hat{\bm{v}}_i^j(\ell)+\hat{\bm{v}}_j^i(\ell)\|_2^2\right) \\
	& \leq \sum_{\ell = 0}^{k}\mathbb{E}\left( \tilde{V}(\ell)-\tilde{V}(\ell+1)\right) \\
	&= \tilde{V}(0)-\mathbb{E}\left(\tilde{V}(k+1)\right)\leq \tilde{V}(0),
	\end{aligned}
	\end{equation*}
	where the last inequality is obtained by dropping the non-positive term $-\mathbb{E}\left(\tilde{V}({k}+1)\right)$. The above inequalities imply that $\{\mathbb{E} (m_i\|\bm{u}_i(k+1)-\bm{u}_i^{\star}\|_2^2)\}$, for all $i\in\mathcal{N}$, is summable and converges to 0. Similarly, $\{ \mathbb{E}(\|\hat{\bm{v}}_i^j(k)-\bm{v}_i^j(k) \|_2^2)\}$, and  $\{\mathbb{E}(\|\hat{\bm{v}}_i^j(k)+\hat{\bm{v}}_j^i(k)\|_2^2)\}$, for all $j\in\mathcal{N}_i$ and $i\in\mathcal{N}$, are also summable and converge to 0. 
	Using the Markov inequality, for any $\varepsilon > 0$,  we have that $\limsup_{k\to\infty} \mathbb{P}\left(\Psi(\bm u, \bm v, \hat{\bm v})\geq \varepsilon\right) \leq \limsup_{k\to\infty} \frac{1}{\varepsilon} \mathbb{E}\left(\Psi(\bm u, \bm v, \hat{\bm v})\right)=0,$ where
	$	 \Psi(\bm{u}, \bm{v}, \hat{\bm v}) 
		= \sum_{i\in\mathcal{N}} m_i\|\bm{u}_i(k+1)-\bm{u}_i^{\star}\|_2^2 
		 +\sum_{i\in\mathcal{N}}  \sum_{j\in \mathcal{N}_i} \left(\frac{3}{2}-\eta_{ij}\right)\|\bm{v}_i^j(k)-\hat{\bm{v}}_i^j(k) \|_2^2 
		+\sum_{i\in\mathcal{N}} \sum_{j\in \mathcal{N}_i} \frac{\eta_{ij}-(2\eta_{ij})^2}{2}\|\hat{\bm{v}}_i^j(k)+\hat{\bm{v}}_j^i(k)\|_2^2.$ 
	\color{black}
	Thus, it holds with probability 1 that
	\begin{align}
		\lim_{k\to\infty}\|\bm{u}_i(k)-\bm{u}_i^{\star}\|_2^2 &= 0, \ \forall i\in\mathcal{N},\label{eq:u-us} \\
		\lim_{k\to\infty}\|\bm{v}_i^j(k)-\hat{\bm{v}}_i^j(k)\|_2^2 &= 0, \ \forall j\in\mathcal{N}_i, \ \forall i\in\mathcal{N},\label{eq:v-vh} \\
		\lim_{k\to\infty}\|\hat{\bm{v}}_i^j(k)+\hat{\bm{v}}_j^i(k)\|_2^2 &= 0, \ \forall j\in\mathcal{N}_i, \ \forall i\in\mathcal{N},\label{eq:vh+vh} 
	\end{align}
	Moreover, based on \eqref{eq:v-vh} and \eqref{eq:vh+vh}, it follows with probability 1 that
	\begin{equation}
		\lim_{k\to\infty}\|{\bm{v}}_i^j(k)+{\bm{v}}_j^i(k)\|_2^2=0,\ \forall j\in\mathcal{N}_i, \ \forall i\in\mathcal{N}.
		\label{eq:vij+vji}
	\end{equation}
	 	 
	 Based on \eqref{eq:ineq_supermart} and the martingale convergence theorem, the sequences $\{\|\bm{v}(k)-\bm{v}^{\star}\|_{\tilde{H}}^2\}$ and $\{\|{\bm{\nu}}(k)-\bm{\lambda}^{\star}\|_{\tilde{H}}^2\}$ are bounded with probability 1, i.e., there exist accumulation points of the sequences $\{\bm{v}(k)\}$ and $\{{\bm{\nu}}(k)\}$. Furthermore, $\{\bm{\lambda}(k)\}$ is also bounded with probability 1 and has accumulation points due to the boundedness of $\{{\bm{\nu}}(k)\}$, the relation in \eqref{eq:vij+vji}, and the fact that ${\bm{z}}_i^j(k)= {\bm{v}}_j^i(k)$, for each $k \in \mathbb{Z}_{\geq 0}$, which follows from the initialization of $\bm{z}_i^j(k)$ in Algorithm \ref{alg_tv} and the update rule \eqref{eq:z_up_agl_tv}.

	 Let $\{(\bm{v}(k_{\ell}), \bm{\lambda}(k_{\ell}))\}$ be a convergent subsequence and assume that  $({\bm{v}}^{\mathrm{a}},\bm{\lambda}^{\mathrm{a}})$ is its limit point. Therefore, due to the initialization of the variables in Algorithm \ref{alg_tv} and the update rules \eqref{eq:z_up_agl_tv} and \eqref{eq:xi_up_agl_tv}, it follows that $\lim_{\ell\to\infty} {\bm{z}}_i^j(k_{\ell})=\lim_{\ell\to\infty} {\bm{v}}_j^i(k_{\ell})=\bm{v}_j^{i\mathrm{a}}$ and $\lim_{{\ell}\to\infty} {\bm{\xi}}_i^j(k_{\ell})=\lim_{_{\ell}\to\infty} {\bm{\lambda}}_j^i(k_{\ell})=\bm{\lambda}_j^{i\mathrm{a}}$ with probability 1, for each $j\in\mathcal{N}_i$ and $i\in\mathcal{N}$. 
	 
	 Now, we need to show that $({\bm{u}}^{\star},{\bm{v}}^{\mathrm{a}},\bm{\lambda}^{\mathrm{a}})$ is a saddle point of $L(\bm{u},\bm{v},\bm{\lambda})$, i.e.,  $({\bm{u}}^{\star},{\bm{v}}^{\mathrm{a}},\bm{\lambda}^{\mathrm{a}})$ satisfies the inequalities in \eqref{eq:saddle_point}. Based on \eqref{eq:vij+vji}, ${\bm{v}}_i^{j\mathrm{a}}+{\bm{v}}_j^{i\mathrm{a}}=\lim_{{\ell}\to\infty}({\bm{v}}_i^j(k_{\ell})+{\bm{v}}_j^i(k_{\ell}))=0$, with probability 1, for all $j\in\mathcal{N}_i$ and $i\in\mathcal{N}$. Thus, we have that, for any $\bm{\lambda}\in\mathbb{R}^{\sum_{i\in\mathcal{N}}n_s|\mathcal{N}_i|}$, $L({\bm{u}}^{\star},{\bm{v}}^{\mathrm{a}},\bm{\lambda}) = L({\bm{u}}^{\star},{\bm{v}}^{\mathrm{a}},\bm{\lambda}^{\mathrm{a}})$, satisfying the first inequality in \eqref{eq:saddle_point}. Now, we show the second inequality in \eqref{eq:saddle_point}. Consider the update step \eqref{eq:primal_uv_up_agl_tv}, for all $i\in\mathcal{N}$, i.e.,
	\begin{align*} 
		({\bm{u}}(k+1),\hat{\bm{v}}(k)) =
		 \arg\min_{(\bm{u}_i,\bm{v}_i) \in \mathcal{C}_i,i\in\mathcal{N}} \sum_{i\in\mathcal{N}}\Bigl(f_i^{\mathrm{p}}(\bm{u}_{i}) + f_i^{\mathrm{s}}(\bm{v}_{i})\Bigr. \notag\\
+		\sum_{j\in \mathcal{N}_i} \left(\langle \bm{\lambda}_i^{j}(k)+\bm{\xi}_i^{j}(k),\bm{v}_{i}^j\rangle+\|\bm{v}_{i}^j + \bm{z}_{i}^{j}(k)\|_2^2\right)\Bigr).
	\end{align*}
	 By substituting $k$ with $k_{\ell}$ and taking the limit as $\ell$ goes to infinity on both sides of the equality, it holds with probability 1 that  
	 \begin{align}
	 ({\bm{u}}^{\star},{\bm{v}}^{\mathrm{a}}) 
	 &= \lim_{{\ell}\to\infty} \arg\min_{(\bm{u}_i,\bm{v}_i) \in \mathcal{C}_i,i\in\mathcal{N}} \sum_{i\in\mathcal{N}}\Bigl(f_i^{\mathrm{p}}(\bm{u}_{i}) + f_i^{\mathrm{s}}(\bm{v}_{i})+\Bigr. \notag\\
	 &\quad \Bigl.\sum_{j\in \mathcal{N}_i} \left(\langle \bm{\lambda}_i^{j}(k_{\ell})+\bm{\xi}_i^{j}(k_{\ell}),\bm{v}_{i}^j\rangle+\|\bm{v}_{i}^j + \bm{z}_{i}^{j}(k_{\ell})\|_2^2\right)\Bigr) \notag\\		
	 &= \arg\min_{(\bm{u}_i,\bm{v}_i) \in \mathcal{C}_i,i\in\mathcal{N}} \sum_{i\in\mathcal{N}}\Bigl(f_i^{\mathrm{p}}(\bm{u}_{i}) + f_i^{\mathrm{s}}(\bm{v}_{i})+\Bigr. \notag\\
	 &\qquad \Bigl.\sum_{j\in \mathcal{N}_i}\left(\langle \bm{\lambda}_i^{j\mathrm{a}}+\bm{\lambda}_j^{i\mathrm{a}},\bm{v}_{i}^j\rangle+\|\bm{v}_{i}^j + \bm{v}_j^{i\mathrm{a}}\|_2^2\right)\Bigr) \notag\\
	 &= \arg\min_{(\bm{u}_i,\bm{v}_i) \in \mathcal{C}_i,i\in\mathcal{N}} \sum_{i\in\mathcal{N}}\Bigl(f_i^{\mathrm{p}}(\bm{u}_{i}) + f_i^{\mathrm{s}}(\bm{v}_{i})+\Bigr. \notag\\
	 &\qquad \Bigl.\sum_{j\in \mathcal{N}_i}\langle \bm{\lambda}_i^{j\mathrm{a}},\bm{v}_{i}^j+\bm{v}_{j}^i\rangle\Bigr). \label{eq:uv_a}
	 \end{align}
	 The left-hand side of the first equality is obtained by using  $\lim_{{\ell}\to\infty}(\bm{u}(k_{\ell}+1),\hat{\bm{v}}(k_{\ell}))=({\bm{u}}^{\star},{\bm{v}}^{\mathrm{a}})$, with probability 1, due to \eqref{eq:u-us} and \eqref{eq:v-vh},  which implies that $\lim_{{\ell}\to\infty}\hat{\bm{v}}(k_\ell)={\bm{v}}^{\mathrm{a}}$, with probability 1. The second equality is obtained since $\lim_{\ell\to\infty} {\bm{z}}_i^j(k_{\ell})=\bm{v}_j^{i\mathrm{a}}$ and $\lim_{\ell\to\infty} {\bm{\xi}}_i^j(k_{\ell})=\bm{\lambda}_j^{i\mathrm{a}}$, with probability 1, for all $j\in\mathcal{N}_i$ and $i\in\mathcal{N}$.  
	Then, the last equality holds since the term  $\sum_{i\in\mathcal{N}}\sum_{j\in \mathcal{N}_i}\|\bm{v}_i^j+\bm{v}_j^{i\mathrm{a}}\|_2^2$ is zero at $({\bm{u}}^{\star},{\bm{v}}^{\mathrm{a}})$ due to the fact that  $\bm{v}_i^{j\mathrm{a}}+{\bm{v}}_j^{i\mathrm{a}}=0$, for all $j\in\mathcal{N}_i$ and $i\in\mathcal{N}$.  Additionally, $\bm{v}^{\mathrm{a}}$ is also an attainer of $\min_{\bm{v}} \sum_{i\in\mathcal{N}}\sum_{j\in \mathcal{N}_i}\|\bm{v}_i^j+\bm{v}_j^i\|_2^2$ since $\bm{v}_i^{j\mathrm{a}}+{\bm{v}}_j^{i\mathrm{a}}=0$, for all $j\in\mathcal{N}_i$ and $i\in\mathcal{N}$. 
	Therefore, the pair $({\bm{u}}^{\star},{\bm{v}}^{\mathrm{a}})$ also minimizes $L(\bm{u},\bm{v},\bm{\lambda}^{\mathrm{a}})$, i.e.,
$		({\bm{u}}^{\star},{\bm{v}}^{\mathrm{a}}) 
		\in \arg\min_{(\bm{u}_i,\bm{v}_i) \in \mathcal{C}_i,i\in\mathcal{N}} \sum_{i\in\mathcal{N}}\Bigl(f_i^{\mathrm{p}}(\bm{u}_{i}) + f_i^{\mathrm{s}}(\bm{v}_{i})+
		\sum_{j\in \mathcal{N}_i}\left(\langle \bm{\lambda}_i^{j\mathrm{a}},\bm{v}_{i}^j+\bm{v}_{j}^i\rangle+\|\bm{v}_{i}^j + \bm{v}_j^{i}\|_2^2\right)\Bigr),$ \color{black}
	where the cost function in the minimization is obtained by adding the quadratic term $\sum_{i\in\mathcal{N}}\sum_{j\in \mathcal{N}_i}\|\bm{v}_i^j+\bm{v}_j^i\|_2^2$ to the cost function on the right-hand side of the last equality in \eqref{eq:uv_a}. 
	Hence, the preceding relation implies the second inequality in \eqref{eq:saddle_point}. 
	Thus, $({\bm{u}}^{\star},{\bm{v}}^{\mathrm{a}},\bm{\lambda}^{\mathrm{a}})$ is a saddle point of $L(\bm{u},\bm{v},\bm{\lambda})$. 	
	Finally, we can set $\bm{v}^{\star}=\bm{v}^{\mathrm{a}}$ and $\bm{\lambda}^{\star}=\bm{\lambda}^{\mathrm{a}}$ in $\tilde{V}(k)$ (see \eqref{eq:Vk_tilde}). Since the subsequence of $\tilde{V}(k_{\ell})$ converges to 0 with probability 1 and $\tilde{V}(k)$ is non-negative supermartingale, the entire sequence $\{(\bm{v}(k),\bm{\lambda}(k)\}$ converges to  $({\bm{v}}^{\mathrm{a}},\bm{\lambda}^{\mathrm{a}})$ with probability 1. \color{black}
\subsection{Proof of Theorem \ref{th:rate}}
\label{sec:pf_th_rate}
	By rearranging the summation of \eqref{eq:ineq_supermart} over $\ell=0,\dots,k-1$ and taking the total expectation, we have that
	\begin{equation}
	\begin{aligned}
	& \sum_{\ell = 0}^{k-1}\sum_{i\in\mathcal{N}} \mathbb{E} \left(m_i\|\bm{u}_i(\ell+1)-\bm{u}_i^{\star}\|_2^2\right)   \\
	& + \sum_{\ell = 0}^{k-1}\sum_{i\in\mathcal{N}}  \sum_{j\in \mathcal{N}_i} \left(\frac{3}{2}-\eta_{ij}\right)\mathbb{E}\left(\|\hat{\bm{v}}_i^j(\ell)-\bm{v}_i^j(\ell) \|_2^2\right)  \\
	& + \sum_{\ell = 0}^{k-1}\sum_{i\in\mathcal{N}} \sum_{j\in \mathcal{N}_i} \frac{\eta_{ij}-(2\eta_{ij})^2}{2}\mathbb{E}\left(\|\hat{\bm{v}}_i^j(\ell)+\hat{\bm{v}}_j^i(\ell)\|_2^2\right) \\
	& \leq \sum_{\ell = 0}^{k-1}\mathbb{E}\left( \tilde{V}(\ell)-\tilde{V}(\ell+1)\right) = \tilde{V}(0)-\mathbb{E}\left(\tilde{V}(k)\right)\\
	&\leq \frac{1}{\underline{\alpha}}\|\bm{v}(0)-\bm{v}^{\star}\|_{{H}}^2+\frac{1}{2\underline{\alpha}}\|\tilde{\bm{\lambda}}(0)-\bm{\lambda}^{\star}\|_{{H}}^2,
	\end{aligned}
	\label{eq:rate_ineq2}
	\end{equation}
	 where the last inequality is obtained by dropping the non-positive term $-\mathbb{E}\left(\tilde{V}(k)\right)$ and by defining $\underline{\alpha} = \min_{\{i,j\}\in\mathcal{E}} \gamma_{i}\gamma_{j}\beta_{ij}$. \color{black}
	 Furthermore, due to the convexity of the squared of the Euclidean norm, it follows that, for $k\geq1$,   
$	 k\mathbb{E} (\|\bar{\bm{u}}_i(k)-\bm{u}_i^{\star}\|_2^2) \leq \sum_{\ell = 0}^{k-1}\mathbb{E} (\|\bm{u}_i(\ell+1)-\bm{u}_i^{\star}\|_2^2),$ 
$	 k\mathbb{E}(\|\bar{\hat{\bm{v}}}_i^j({k-1})-\bar{\bm{v}}_i^j({k-1}) \|_2^2) \leq \sum_{\ell = 0}^{k-1}\mathbb{E}(\|\hat{\bm{v}}_i^j(\ell)-\bm{v}_i^j(\ell) \|_2^2),$ 
$	 k \mathbb{E}(\|\bar{\hat{\bm{v}}}_i^j(k-1)+\bar{\hat{\bm{v}}}_j^i(k-1)\|_2^2) \leq \sum_{\ell = 0}^{k-1} \mathbb{E}(\|\hat{\bm{v}}_i^j(\ell)+\hat{\bm{v}}_j^i(\ell)\|_2^2)$. \color{black}
	 By applying the above relations to \eqref{eq:rate_ineq2} and using the fact that $m_i>0$, for all $i\in \mathcal{N}$, and $\frac{3}{2}-\eta_{ij}>0$, $\frac{\eta_{ij}-(2\eta_{ij})^2}{2}>0$, for all $j\in\mathcal{N}_i$ and $i\in\mathcal{N}$, we have the desired convergence rate, i.e.,  for $k\geq 1$,
	 	\begin{align}
	&\sum_{i\in\mathcal{N}}\mathbb{E} \left(m_i\|\bar{\bm{u}}_i(k)-\bm{u}_i^{\star}\|_2^2\right) \notag\\
	&+\sum_{i\in\mathcal{N}}  \sum_{j\in \mathcal{N}_i} \left(\frac{3}{2}-\eta_{ij}\right)\mathbb{E}\left(\|\bar{\hat{\bm{v}}}_i^j(k-1)-\bar{\bm{v}}_i^j(k-1) \|_2^2\right) \notag\\
	&+\sum_{i\in\mathcal{N}} \sum_{j\in \mathcal{N}_i} \frac{\eta_{ij}-(2\eta_{ij})^2}{2}\mathbb{E}\left(\|\bar{\hat{\bm{v}}}_i^j(k-1)+\bar{\hat{\bm{v}}}_j^i(k-1)\|_2^2\right)  \notag\\
	&\leq  \frac{1}{\underline{\alpha}k}\Big(\Big.\|\bm{v}(0)-\bm{v}^{\star}\|_{{H}}^2+\frac{1}{2}\|\tilde{\bm{\lambda}}(0)-\bm{\lambda}^{\star}\|_{{H}}^2\Big.\Big).  \label{eq:rate}
	\end{align}
\begin{remark}
 The inequality \eqref{eq:rate} implies that if the activation probabilities of agents and links are larger, then the convergence is achieved faster.  \eod \color{black}
\end{remark}


\section{Conclusion and Future Work}
\label{sec:concl}
This technical note discusses a distributed algorithm for a multi-agent convex optimization problem with edge-based coupling constraints, which is related to energy management problems. The proposed method works asynchronously over time-varying communication networks. We model the asynchronicity and the time-varying nature of the communication network as random processes and show the convergence and the rate of the proposed algorithm. 
As future work, we consider generalizing the problem that can be dealt with, such as by introducing global objectives of control, coupling inequality constraints, or non-convex coupling constraints, which is relevant to the optimal power flow problems in power systems. Moreover, we also consider the implementation of inexact minimization to the proposed algorithm to reduce computational burden.


\bibliographystyle{ieeetran}
\bibliography{ref_ja}

\end{document}